\documentclass[11pt,fleqn]{amsart} % \documentclass[11pt]{article}

\usepackage[usenames,dvipsnames]{xcolor}

\usepackage{amsthm}
\usepackage{amsfonts}
\usepackage[english]{babel}
\usepackage[usenames]{xcolor}
\usepackage{graphicx}
\usepackage{soul}
\usepackage{stfloats}
\usepackage{morefloats}
\usepackage{cite}
\usepackage{lscape}
\usepackage{epstopdf}
\usepackage{braket}
\usepackage[lite]{amsrefs}
\usepackage{mathbbol}
\usepackage{tikz}
\usepackage{shuffle}

\usepackage{algorithm,algorithmicx,algpseudocode}

\usepackage{amsmath,amstext,amsopn,amsfonts,eucal,amssymb}
\usepackage{graphicx,wrapfig,url}

\newtheorem{theorem}{Theorem}[section]

\newtheorem{lemma}[theorem]{Lemma}

\newtheorem{definition}[theorem]{Definition}

\newtheorem{remark}[theorem]{Remark}

\newtheorem{question}[theorem]{Question}

\begin{document}

	\title{Spectral methods for Neural Integral Equations}

	\author{Emanuele Zappala} 
	\address{Department of Mathematics and Statistics, Idaho State University\\
		Physical Science Complex |  921 S. 8th Ave., Stop 8085 | Pocatello, ID 83209} 
	\email{emanuelezappala@isu.edu}

	\maketitle

	\begin{abstract}
		Neural integral equations are deep learning models based on the theory of integral equations, where the model consists of an integral operator and the corresponding equation (of the second kind) which is learned through an optimization procedure. This approach allows to leverage the nonlocal properties of integral operators in machine learning, but it is computationally expensive. In this article, we introduce a framework for neural integral equations based on spectral methods that allows us to learn an operator in the spectral domain, resulting in a cheaper computational cost, as well as in high interpolation accuracy. We study the properties of our methods and show various theoretical guarantees regarding the approximation capabilities of the model, and convergence to solutions of the numerical methods. We provide numerical experiments to demonstrate the practical effectiveness of the resulting model. 
	\end{abstract}
	
	
	
	\date{\empty}

	\tableofcontents

	\section{Introduction}
	
	The theory of integral equations (IEs) has found important applications in several disciplines of science. In physcics, for example, integral operators and integral equations are found in the theory of nonlocal gravity \cite{Nonlocal_Grav}, plasma physics \cite{Plasma_IE}, fluid dynamics \cite{Fluid}, and nuclear reactor physics \cite{Adler}. Several applications are also found in engineering \cite{Eng_IE}, brain dynamics \cite{Amari,WC}, and epidemiology \cite{Epid}. Such a wide range of applications has therefore motivated the study of IE and integro-differential equation (IDE) models in machine learning, for operator learning tasks \cite{NIDE,ANIE,ANIE_NAT}. These models are called neural IEs (NIEs) and neural IDEs (NIDEs), respectively. They need to solve an IE or IDE at each step of the training (and for each data sample), therefore incurring in significant computational costs. While there is a benefit in modeling dynamics using neural IEs and IDEs due to the nonlocality of the integral operators, it is of interest to obtain formulation of these models that are more computationally efficient. 
	
	Operator learning is a branch of machine learning where the training procedure aims at obtaining an operator acting in an infinite dimensional space of functions. The theoretical foundation of operator learning has been significantly developed in \cite{DeepONet}, based on results dating back to the 90's, and more recently found examples in several other works, e.g. \cite{kovachki2023neural,gupta2021multiwavelet,hao2023gnot,lin2021operator}. Operator learning allows to learn the governing equations of a system from data alone, without further assumptions or knowledge on the properties of the system that generated the dataset. As such, it is greatly useful in studying complex systems whose theoretical properties are not completely understood. 
	
	The scope of this article is to introduce a framework for neural IEs, which is an operator learning problem, based on spectral methods. This allows to perform integration in the spectral domain, greatly simplifying the architecture of the models, but still retaining their great modeling capabilities of NIEs, especially in regards to long range dependencies (high nonlocality). 
	
	The convenience of this method with respect to the approach to neural network approaches for learning integral operators in IEs and IDEs found in \cite{ANIE,ANIE_NAT,NIDE} lies in the following facts:
	\begin{itemize}
		\item 
		A small (in terms of parameters) neural network $G_\theta$ can achieve comparable expressivity and task accuracy to larger models;
		\item 
		Integration in this setting consists only of a matrix multiplication as opposed to quadrature rules or Monte Carlo methods, therefore improving computational speed and memory scalability of the model;
		\item 
		Th output of the spectral IE solver implemented in this article is a solution $\mathbf y$ which is expanded in the Chebyshev basis, and that can therefore be easily interpolated. 
	\end{itemize} 
	Overall, this method produces highly smooth outputs that can be interpreted as denoised approximations of the dataset. Throughout this article we will use Chebyshev polynomials of the first kind, referring to them simply as Chebyshev polynomials without further specification. 
	
	The codes implemented for this method, and for the experiments found in Section~\ref{sec:experiments}, can be found at \url{https://github.com/emazap7/Spectral_NIE}.

	\section{Methods}\label{sec:Mehtods}
	
	In this section we introduce and discuss the spectral approaches used in this article for integral operator learning. We consider the one dimensional case, where our method is based on Chebyshev polynomials, and defer a generalization of this methods to higher dimensions to a later work. 
	
	We consider integral equations of Fredholm and Volterra types which take the form
	\begin{eqnarray}\label{eqn:IE1D}
	\mathbf y(t) = f(\mathbf y,t) + \lambda\int_{-1}^{\alpha(t)} G(\mathbf y, t,s)ds,
	\end{eqnarray}
	where $\alpha(t) = t$ for a Volterra equation, and $\alpha(t) = 1$ for a Fredholm equation, $G : \mathbb R\times [-1,1] \times [-1,1] \longrightarrow \mathbb R$ is potentially non-linear function in $\mathbf y$ called the {\it kernel}, and $\mathbf y: [-1,1] \longrightarrow \mathbb R$ is the unknown function. In this article we consider equations defined on the interval $[-1,1]$, which are sufficient for the scope of machine learning in the assumption that the time intervals have been normalized. In the context of traditional IEs (with no learning), spectral approaches for linear equations in $\mathbf y$, where $G(\mathbf y,t,s) = K(t,s)\mathbf y$ have been treated in \cite{Liu}, while the nonlinear case where $G$ has the form $G(\mathbf y,t,s) = K(t,s)\mathbf y^m$ for some $m>1$ has been considered in \cite{Yang}. In the present work, we consider arbitrary functions $G$ where the learning process determines a neural network $G_\theta$ which is obtained through optimization on a given dataset. Our setup is an operator learning problem where we want to learn the integral operator by learning the parameters $\theta$ of $G_\theta$, therefore obtaining
	\begin{eqnarray}\label{eqn:NIE1D}
	\mathbf y(t) = f(\mathbf y,t) + \lambda\int_{-1}^{\alpha(t)} G_\theta(\mathbf y, t,s)ds,
	\end{eqnarray}
	
	Let us denote by $C_j(t)$ the $j^{\rm th}$ Chebyshev polynomial of the first kind, and let $N$ be a fixed natural number. We write the truncated Chebyshev series as 
	\begin{eqnarray}\label{eqn:Cheb_series}
	\mathbf y(t) \approx \mathbf y_N(t) := \sum_{j=0}^N a_j C_j(t),
	\end{eqnarray}
	where $a_j\in \mathbb R$ are coefficients that determine the expansion of $\mathbf y$ in $C_j(t)$. The collocation points $\{t_k\}_{k=0}^N \subset [-1,1]$ are optimally taken to be the Chebyshev collocation points $t_k = \cos(\frac{k\pi}{N})$ for $k=0, \ldots, N$, but we will experimentally show that the method is effective even when the points $t_k$ are taken to be different from this choice, which is of importance since datasets might have arbitrary time stamps (including irregularly sampled time points).  
	
	Using Equation~\eqref{eqn:Cheb_series}, from Equation~\eqref{eqn:IE1D} we obtain
	\begin{eqnarray}\label{eqn:SIE1D}
	\mathbf y(t_i) = f(\mathbf y(t_i),t_i) + \lambda\int_{-1}^{\alpha(t_i)} G(\sum_{j=0}^N a_j C_j(s),t_i,s)ds,
	\end{eqnarray}
	which we want to solve for all $i=0, \ldots, N$. Applying this setup to the case of integral operators parameterized by neural networks, we find the equation corresponding to Equation~\eqref{eqn:SIE1D} 
	\begin{eqnarray}\label{eqn:SNIE1D}
	\mathbf y(t_i) = f(\mathbf y(t_i),t_i) + \lambda\int_{-1}^{\alpha(t_i)} G_\theta(\sum_{j=0}^N a_j C_j(s),t_i,s)ds.
	\end{eqnarray}
	In order to solve Equation~\eqref{eqn:SIE1D} or Equation~\eqref{eqn:SNIE1D}  we need an efficient way of computing the integrals $\int_{-1}^{\alpha(t_i)} G(\sum_{j=0}^N a_j C_j(s),t_i,s)ds$, and $\int_{-1}^{\alpha(t_i)} G_\theta(\sum_{j=0}^N a_j C_j(s),t_i,s)ds$ for all $i = 0, \ldots, N$. Let us first consider the former case, where $\theta$ does not appear. For the moment we will only concern ourselves with being able to solve Equation~\eqref{eqn:SIE1D}, and Equation~\eqref{eqn:SNIE1D} for given parameters $\theta$, without concerning ourselves with optimizing the parameters $\theta$. 
	
	Observe that for each fixed spectral collocation point $t_i$, and fixed $\mathbf y$ (hence fixed coefficients $a_k$), the function $G$ induces a function $\hat G^i : [-1,1] \longrightarrow \mathbb R$ obtained through the correspondence $s\mapsto G(\sum_{j=0}^N a_j C_j(s),t_i,s)$. Here we suppress the dependence of $\hat G^i$ on $\mathbf y$ for notational simplicity. Similarly to how $\mathbf y$ has been expanded through a truncated Chebyshev series we have
	\begin{eqnarray}
	\hat G(s) \approx \hat G_N(s) :=  \sum_{r=0}^N b_r C_r(s),
	\end{eqnarray}
	for some coefficients $b_r$. In the assumption that the $b_r$ are known, there is a computationally efficient method for computing the integral $\int_{-1}^{\alpha(t_i)} \hat G^i(s)ds$ through the coefficients of the Chebyshev expansion, given in \cite{Got-Ors,Greengard}. We recall these results for Fradholm and Volterra equations here. 
	
	First, let us denote by $\frak X$ the space of infinite sequences $(a_0, \ldots , a_n, \ldots)$ of real numbers. We can express a function $\mathbf y : [-1,1] \longrightarrow \mathbb R$ as such a sequence by associating to $\mathbf y$ its Chebyshev coefficients. In practice, we are interested in sequences that vanish after a certain $N\in \mathbb N$, since we consider in this article finite Chebyshev expansions.  As such, we can express integral (Fredholm and Volterra) operators as mappings $\frak I: \frak X \longrightarrow \frak X$. Let $\mathbf b \in \frak X$, and let $\hat{\mathbf b}$ denote the sequence obtained from $\mathbf b$ by defining $\hat b_0 := 2b_0$, and setting $\hat b_n = b_n$ for all $n>0$. 
	
	Then, (see \cite{Greengard,Got-Ors}) we define $\mathbf d := \frak I(\mathbf b)$ by the formula
	\begin{eqnarray}
	d_k = \frac{\hat b_{k-1}-\hat b_{k+1}}{2k},
	\end{eqnarray}
	for all $k>1$, and
	\begin{eqnarray}
	d_0 = \sum_{i=1}^{\infty} (-1)^{i+1} d_i. 
	\end{eqnarray}
	In this article, in practice, the latter is finite, as all the coefficients $a_k$ are eventually zero (we only consider finite Chenyshev expansions), and we therefore will start using finite sums up to $N$, to explicitly take this fact into account. This mapping has been called spectral integration in \cite{Greengard}, and we will follow this convention here. We can now define the Fredholm and Volterra integral operators in spectral form, $\mathcal I_F$ and $\mathcal I_V$ respectively,  by the formulae
	\begin{eqnarray}
	\int_{-1}^1 \hat G^i(s)ds = 2\sum_{k=0}^N d_{2k+1}  \label{eqn:Spectral_int_Fredholm}\\
	\int_{-1}^{\alpha(t_i)} \hat G^i(s)ds = \frac{1}{2} d_0 + \sum_{k=1}^N d_kC_k(\alpha(t_i)), \label{eqn:Spectral_int_Volterra}
	\end{eqnarray}
	for all spectral collocation points $t_i$, cf. Lemma~2.1 in \cite{Greengard} and the Appendix to \cite{Got-Ors}. 
	
	Solving the IE problem in Equation~\eqref{eqn:SIE1D} presents now the complication that once we have the Chebyshev expansion of $\mathbf y$, i.e. its coefficients $a_n$, $n\leq N$, we need to compute the Chebyshev coefficients $b_n$ of $G^i$ for each collocation point. This procedure requires an integration per collocation point in order to project the function $G$ onto the Chebyshev basis. However, when our task is to learn the integral operator as in the present article, we can bypass the issue altogether, and learn the neural network $G_\theta$ as a mapping in spectral space, directly giving the correspondence $\mathbf a \mapsto \mathbf b$. Moreover, integration in this formulation simply accounts to performing a matrix multiplication, as seen in Equation~\eqref{eqn:Spectral_int_Fredholm} and Equation~\eqref{eqn:Spectral_int_Volterra}.

	%\subsection{Higher dimensional case}
	
	\section{Theoretical framework and approximation capabilities}\label{sec:Theoretical}
	
	 A fundamental reason to formulate an operator learning problem on a finite space, such as the space spanned by the Chebyshev polynomials of degree at most $n$, is that it allows us to effectively perform the optimization process on a finite dimensional approximation of the original infinite dimensional space. In order to guarantee that the algorithm produces solutions close to the target ones, it is important to suitably choose the finite dimensional representation of the orginal space. 
		
	In this article, we use Chebyshev polynomials, which are known to provide stability and rapid convergence in approximation problems. For example, they are known to satisfy an optimality condition on the approximation of the degree $n+1$ monomial via lower degree polynomials over the closed interval $[-1,1]$. Moreover, the normalized polynomial $C_n$ realizes the minimum uniform norm among the deegree $n$ monic polynomials. As a result, when minimizing a function with uniformly bounded $n+1$ derivative, Chebyshev interpolation at the Chebyshev spectral collocation points minimizes the error. These properties result in a decreased Runge phenomenon outside of the interpolation points, therefore reducing oscillations in interpolation problems.  
	
	After projecting a function on a truncated polynomial basis, an integral operator $T$ applied to this function does not generally produce another truncated polynomial, and the explicit form obtained needs to be determined based on the form of the kernel of the operator. However, we can subsequently project the output again to consider the restriction of the operator on the space of Chebyshev polynomials with fixed highest degree. This process effectively produces a representation of the operator on a finite dimensional space, which is the target of our learning algorithm. We will show that, under some relatively mild assumptions, we can learn a neural network that approximates the finite dimensional representation of an integral operator, and such that the solutions of the corresponding equation in finite space approximate the solutions of the integral equation in the infinite dimensional space.  
	
	We can reformulate the method as described in Section~\ref{sec:Mehtods} using a more general language. While this does not add any new information to the setup, it makes a general treatment simpler. In this regard, we let $X$ denote a Banach space of functions of interest (whose precise definition depends on the problem considered), and we indicate the integral operator of Equation~\eqref{eqn:IE1D} by $T : X \longrightarrow X$, where we assume that a choice of $\alpha$ is performed. 
	
		\begin{definition}
				{\rm 
						Let $f: [a,b] \longrightarrow \mathbb R$ be a function. Then, we say that $f$ is {\it H\"older continuous} if there exists a constant $c$ such that 
						\begin{eqnarray*}
								|f(x) - f(y)| &\leq& c|x-y|^\beta ,
						\end{eqnarray*}
						for all $x,y\in [a,b]$. The case where $\beta=1$ is of particular importance, and an H\"older function of exponent $\beta = 1$ is called {\it Lipschitz continuous}. We define the set $C^{k,\beta}([a,b])$ of all H\"older functions whose $k$th derivative is H\"older continuous, and define the norm 
						\begin{eqnarray*}
								\|f\|_{k,\beta} &=& \|f\|_k + \sup\{\frac{|f^k(x) - f^k(y)|}{|x-y|^\beta} \ |\ x,y\in [a,b], x\neq y \},  
						\end{eqnarray*}
						where $\|f\|_k := \max_{r\leq k} \|f^r\|_\infty$ is the maximum among the uniform norms of the derivatives of $f$. It is well known that for any $0<\beta \leq 1$, the space $C^{k,\beta}([a,b])$ is a Banach space with the norm $\|\cdot \|_{k,\beta}$. Morevoer, $C^{k,\gamma}([a,b])$ compactly embeds into $C^{k,\beta}([a,b])$ whenever $1 \geq \gamma > \beta >0$.  
				}
		\end{definition}
	
	In this article, the space $X$ is assumed to be a H\"older space, and we will consider some suitable subspace of $X$ for our theoretical guarantees. 
	When the integral operator is parameterized by a neural network with parameters $\theta$, we indicate it by $T_\theta$ when we want to explicitly make this fact clear. 
	
	We observe that the Chebyshev expansion considered in Equation~\eqref{eqn:Cheb_series} is the result of projecting a function $y\in X$ on the space $X_N$ spanned by the first $N$ Chebyshev polynomials. So, we have a projector (or projection operator) $P_N : X \longrightarrow X_N$, the precise formulation of which will be given in the algorithm. From now on, we will consider a situation where the operator $T$ maps in a closed subspace $R$ of $X$ such that for each $y\in R$ the projections $P_Ny := y_N$ satisfy $\|y - y_N\| \longrightarrow 0$ as $N \longrightarrow \infty$. 
	
	The method of Section~\ref{sec:Mehtods} consists of approximating the integral equation $(\mathbb 1 - T)y = f$ through a function and a spectral integration on the Chebyshev space $X_N$, and solving the corresponding integral equation, which we denote by
	\begin{eqnarray}\label{eqn:IE_Cheb_proj}
	(\mathbb 1 - P_N  \frak IT_{\theta, N})y_N = P_N f,
	\end{eqnarray}
	where $\frak I$ denotes the spectral integration (either Fredholm, or Volterra), and $T_{\theta, N}$ indicates a neural network parameterized by $\theta$ and that is defined over the space $X_N$. We consider now also the related problem of ``projecting'' the IE over $X_N$ by means of the projection $P_N : X \longrightarrow X_N$, for a choice of $N >0$
	\begin{eqnarray}\label{eqn:IE_N_proj}
	(\mathbb 1- P_NT) y_N = P_N f
	\end{eqnarray}
	in the space $X_N$. See also Chapter 12 in \cite{Atk-Han} for a similar procedure. The underlying idea of the method is that since the spaces $X_N$ approximate increasingly well the given subspace of $X$, solving Equation~\eqref{eqn:IE_N_proj} for $N$ large enough produces an approximate solution $y_N$ which is close enough to the solution $y$ of Equation~\eqref{eqn:IE1D}. Analogously, when using a parameterized operator (as in Equation~\eqref{eqn:NIE1D}), we can solve with high accuracy when $N$ is large enough.
	
	Some questions now arise.
	\begin{question}\label{question1}
		{\rm 
			Firstly, it is of interest to understand whether if $u_N\in X_N$ is a solution of Equation~\eqref{eqn:IE_N_proj} for each given $N$, we have $\|y - u_N\| \longrightarrow 0$ as $N\longrightarrow \infty$, where $y$ is the real solution of the IE before projecting it. We observe that while this is naturally expected (cf. with the intuitive idea above), it is not necessarily true.
		}
	\end{question} 
	
	\begin{question}\label{question2}
		{\rm
			Secondly, another important question regards the possibility of approximating projected operators $P_NT : X_N \longrightarrow X_N$ through the integral operators in spectral domain described in Section~\ref{sec:Mehtods}.  The importance of this question regards the fact that if $T$ is such that $\|y - u_N\| \longrightarrow 0$ as $N\longrightarrow \infty$, then we can approximate solutions of Equation~\eqref{eqn:IE1D} by knowing solutions of Equation~\eqref{eqn:IE_N_proj}, and approximating $P_NT$ would mean that we can approximate the original problem through a neural network in the spectral space with high accuracy (upon choosing $N$ large enough).
		}
	\end{question} 
	
	\begin{question}\label{question3}
		{\rm 
			Thirdly, it remains to be understood whether when the first and second problems have positive answer, it is possible to learn the operators $P_NT$, e.g. through Stochastic Gradient Descent, or adjoint methods as in \cite{NIDE}. 
		}
	\end{question}
	
	We start by considering Question~\ref{question1} where we see that in some well-behaved situations solving the projected equation gives an approximate solution of the original equation. We recall that an operator $T: X \longrightarrow X$ is said to be bounded if $\|T(y)\| \leq M \|y\|$ for all $y\in X$ and some $M$. While boundedness and continuity coincide for linear operators, the two concepts are different when dealing with nonlinear operators, as in our situation.
	
	We assume that the integral equation $T(y) +f = y$ admits a solution $u$ that is isolated and of non-zero index for $T$ (see \cite{Topological}), where $u$ is a H\"{o}lder continuous function that belongs to $C^{k,\beta}$ for some $k$ and $\beta < \frac{1}{2}$. We assume that $T$ is completely continuous and Frechet differentiable, and that $T'$ is bounded, with bound $L$. These assumptions are not very restrictive, see for instance \cite{Topological} for several families of integral operators satisfying them. Also, in what follows, we assume that $T$ maps the H\"{o}lder space $C^{k,\beta}$ into itself. This simply means that $T$ takes regular enough functions and does not make them less regular. Here we set $D$ to denote a closed bounded domain of $C^{k,\beta}$ containing $u$, such that $T(D) \subset C^{k,\gamma}$ for some $\gamma > 2\beta$, $\inf \|T(D)\| >0$ and $\sup \|T(D)\|_\gamma < \infty$. We let $R$ denote the smallest closed subspace containing the image $T(D) \subset R$, and assume that the $P_N$'s are uniformly bounded on $R$. Lastly, we assume that $1$ is not an eigenvalue of $T'(u)$, and therefore $(\mathbb 1 - T'(u))^{-1}$ exists and it is bounded. The results found in \cite{Pot_Ury,Lanza} give conditions under which examples of such operators can be given, and show that under relatively mild regularity assumptions on the kernel of the operator, a large family of examples arises. Carath\'eodory-type functions, for instance, produce a class of operators of interest.
	
	\begin{theorem}\label{thm:regular_T}
		In the hypotheses above, the projected equation 
		$$
		P_NT(y_N) + f = y_N
		$$
		admits at least a solution $u_N$ for all $N$ large enough. Moreover, 
		$$
		\|u - u_N\| \longrightarrow 0
		$$
		as $N \longrightarrow \infty$, with the same rate as $\|P_Nu - u\| \longrightarrow 0$. 
	\end{theorem}
	\begin{proof}
		For $y$ in $C^{k,\gamma}$, Jackson's Theorem gives us that 
		$$
		\|y - P_Ny\|_\infty \leq \frac{M_k(y)}{n^{k+\gamma}} \longrightarrow 0,
		$$
		where $M_k(y)$ is a constant linearly related to the H\"older constant of the $k^{\rm th}$ derivative. 
		Applying Kalandiya's Theorem \cite{Kal,Ioak}, the convergence in uniform norm implies convergence in H\"older norm. So, we find that
		$$
		\|y - P_Ny\| \longrightarrow 0,
		$$
		for all $y\in R$, using the fact that $P_N$ are uniformly bounded on $R$. 
		From the fact that the operator $T$ has an isolated fixed point of non-zero index, we deduce that we can apply the framework of \cite{Topological} Chapter 3 (see also \cite{Geometrical}). This gives us that the projected equations eventually admit solutions $u_N$, and that the projected solutions converge to $u$, the fixed point of the operator $T + f$. This facts do not depend on the function $f$ appearing in the equation, as one can easily verify directly. Now, we want to show that the numerical scheme converges to $u$ uniformly, at the same rate as $\|P_Nu - u\|$. To do so, we proceed as in \cite{Atk-Pot}. We write the equality 
		$$
		(\mathbb 1 - T'(u))(u_N - u) = (P_N-\mathbb 1)T'(u)(u_N-u) + (P_N - \mathbb 1)u + P_N[T(u_N) - T(u) - T'(u)(u_N-u)].
		$$
		We then derive the inequalities
		\begin{eqnarray*}
			\|u_N - u\| &\leq& \|(\mathbb 1 - T'(u))^{-1}T'(u)(u_N-u)\| + \|(\mathbb 1 - T'(u))^{-1}(P_N-\mathbb 1)u\| \\
			&& + \|(\mathbb 1 - T'(u))^{-1}P_N[T(u_N) - T(u) - T'(u)(u_N - u)]\|\\
			&\leq& \|(\mathbb 1 - T'(u))^{-1}\|\cdot \|T'(u)\|\cdot \|u_N - u\|+ \|(\mathbb 1 - T'(u))^{-1}\|\cdot \|P_N u - \mathbb 1 u\| \\
			&& \frac{\|(\mathbb 1 - T'(u))^{-1}P_N(T(u_N) - T(u) - T'(u)(u_N-u)\|}{\|u_N-u\|}\cdot \|u_N - u\|.
		\end{eqnarray*}
		The fraction $q_N = \frac{\|(\mathbb 1 - T'(u))^{-1}P_N(T(u_N) - T(u) - T'(u)(u_N-u)\|}{\|u_N-u\|}$ goes to zero, as $N$ goes to $\infty$ from the results of \cite{Topological,Geometrical}.

		We have, as noted above, $\|P_Nu-u\| \longrightarrow 0$ as $N\longrightarrow 0$. In addition
		we have $\|(P_N - \mathbb 1)T'(u)\| \longrightarrow 0$ over $D$ and we can bound $\|P_N\|$ uniformly over $D$. We therefore obtain
		\begin{eqnarray*}
			\|u_N-u\| \leq \frac{1}{1- \|(P_N-\mathbb 1)T'(u)\| - \|P_N\|q_N}\|P_Nu - u\|,
		\end{eqnarray*}
		where $\|(P_N-\mathbb 1)T'(u)\| \longrightarrow 0$ and $\|P_N\|q_N \longrightarrow 0$, and therefore for $N$ large enough this inequality makes sense. We obtain that $\|u_N-u\|$ goes to zero with at least the same speed as $\|P_Nu - u\|$, whose convergence to zero depends on $k$ and $\beta$, as shown above. We can also get a similar lower bound following an analogous procedure, showing that the speed with which $\|u_N - u\|$ converges to zero is the same as the speed with which $\|P_N u - u\|$ converges to zero. 
	\end{proof}

	\begin{remark}
		{\rm 
			Observe that since $\|P_N\| \longrightarrow \infty$ as $N \longrightarrow \infty$ when $P_N$ is applied on functions that are not regular enough, e.g. outside of $R$, the previous result does not give that $\|u-u_N\| \longrightarrow 0$ always. In fact, from the Unifrom Boundedness Principle, we know that there must be functions for which $\|u-P_Nu\|$ does not converge to zero. In such cases, the same procedure above cannot be applied. This is somehow inconvenient, but not unexpected. In fact, the Chebyshev expansion is known to not always converge. However, it converges for functions that are regular enough, which is the most important case in applications, especially in physics and engineering.  
		}
	\end{remark}
	
	\begin{remark}
		{\rm 
			The result above is local in nature, in the sense that we are considering a ball $D$ in the space $C^{k,\beta}$. The ball has arbitrary radius, as long that it contains the solution $u$ to the integral equation, which is not a very strong condition. When $D$ contains more than a single fixed point, one can restrict the ball around a fixed point that does not contain other fixed points (using the assumption that a solution to the equation is isolated). 
		}
	\end{remark}

	There is a consideration regarding the method of Section~\ref{sec:Mehtods} in relation to the result of Theorem~\ref{thm:regular_T}. In fact, as in any collocation method, the spectral approach followed in this article produces solutions at the collocation points $\{t_i\}$, while in Theorem~\ref{thm:regular_T} the functions $u_N$ are assumed to be solutions for all $t$ (i.e. not only at the collocation points). We momentarily leave aside the further issue that we only obtain approximate solutions at the collocation points. As a consequence, we can consider that for each $N$ we obtain an interpolation of the function $u_N$. However, we observe that it is known that for a function $\xi$ which is $N+1$ times continuously differentiable, the error for a Chebyshev interpolation $\Pi_N$ of degree $N+1$ is bounded as
	$$
	\|\xi - \Pi_N\|_\infty \leq \frac{K_N}{2^N(N+1)!},
	$$
	where $K_N = \max_{[-1,1]} |\xi^{N+1}|$, and the points $t_i$ of interpolation are assumed to be the Chebyshev nodes. As a consequence, we can obtain solutions $u_N$ with arbitrarily high precision through the described collocation method. We point out that, in general, a dataset might require points $t_i$ that do not coincide with the Chebyshev nodes, in which case the polynomial interpolation would not be ``optimal''. We show that our algorithm still works in such situations with good accuracy by means of experimentation. See Section~\ref{sec:experiments}, where  the data points are not assumed to be at the Chebyshev nodes.
	
	We now turn to Question~\ref{question2}. We assume here that the operator $T:  X\longrightarrow X$ is defined through a kernel $G$ that is continuous in all three variables $y, t$ and $s$, and it is therefore a continuous operator with respect to the uniform norm. Here $X$ is a space of continuous functions depending on the problem (e.g. continuous functions with some fixed boundary in $[-1,1]$).
	
	\begin{lemma}\label{lem:spec_approximation}
		With notation and assumptions as above, let $T$ denote a Fredholm or Volterra integral operator, and let $P_NT : X_N \longrightarrow X_N$ be its projection on $X_N$. Then, for any fixed $\epsilon > 0$, there exist neural networks $F_N, g_N: X_N \longrightarrow X_N$ that satisfiy the condition
		$$
		\|P_NT(y) - P_N \frak I F_N(y) - g_N(y)\|_\infty < \epsilon,
		$$
		for all $y\in X_N$, where $\frak I$ denotes the spectral integration in Fredholm or Volterra form (depending on $T$). Moreover, $P_N \frak I F_N$ approximates with arbitrarily high precision the operator defined through: $P_N\int_{-1}^{\alpha(t)} P_NG(y,t,s)ds$. 
	\end{lemma}
	\begin{proof}
		We focus on the case of Fredholm operators, but the reasoning given here can be applied also to the case of Volterra operators with minor modifications. 
		
		Since $y\in X_N$, it follows that any such function can be expressed as a linear combination $y(s) = \sum_{i=0}^N a_i C_i(s)$. Therefore, the operator $T(y)$ induces a function $\frak T: \mathbb R^{N+1}\times [-1,1] \longrightarrow \mathbb R$ through the assignment 
		$$
		(a_0, \ldots , a_N, t) \mapsto T(\sum_i a_i C_I)(t). 
		$$
		Since $T$ is continuous with respect to the uniform norms, we can show that $\frak T$ is continuous as well, where the domain is endowed with the standard Euclidean norm. In fact, for a tuple $(a_{0,1}, \ldots , a_{N,1},t_1)\in \mathbb R^{N+1}\times [-1,1]$, and an  arbitrary choice of $\epsilon >0$, from the uniform continuity of $T$ we can find $\delta'$ such that $\max_s |y_1(s) - y_2(s)| < \delta'$ implies that $\max_t |T(y_1)(t) - T(y_2)(t)| < \epsilon/2$, where $y_k := \sum_i a_{i,k} C_i$, for $k=1,2$.  Then, taking a ball $B = \mathbb B_{\frac{\delta'}{N+1}}(a_{0,1}, \ldots , a_{N,1}, t_1)$ of radius $\frac{\delta'}{N+1}$ centered in $(a_{0,1}, \ldots , a_{N,1}, t_1)$ we have that $|y_1(s) - y_2(s)| = |\sum_i (a_{i,1}-a_{i,2})C_i(s)| \leq \sum_i |a_{i,1} - a_{i,2}| < \delta'$, whenever $(a_{0,2}, \ldots, a_{N,2},t_2)$ is taken in $B$. This in turn implies that $|T(y_1)(t) - T(y_2)(t)| < \epsilon/2$ for all $t\in [-1,1]$. Since $T(y_1)$ is a continuous function with respect to $t$, we can find $\delta''$ such that $|T(y_1)(t_1) - T(y_1)(t_2)| < \epsilon/2$ whenever $|t_1 - t_2|<\epsilon/2$. Choosing $\delta = \min \{\frac{\delta'}{N+1}, \delta''\}$, we have
		\begin{eqnarray*}
			|T(y_1)(t_1) - T(y_2)(t_2)| &=& |T(y_1)(t_1)- T(y_1)(t_2) + T(y_1)(t_2)  - T(y_2)(t_2)|\\
			&\leq& |T(y_1)(t_1)- T(y_1)(t_2)| + |T(y_1)(t_2)  - T(y_2)(t_2)|\\
			&<& \epsilon/2 + \epsilon/2 = \epsilon,
		\end{eqnarray*}
		whenever $(a_{0,2}, \ldots , a_{N,2}, t_2) \in  \mathbb B_{\delta}(a_{0,1}, \ldots , a_{N,1}, t_1)$. Using the definition of $\frak T$, this implies that $\frak T$ is continuous as claimed. 
		
		Now, using the spectral integration (Section~\ref{sec:Mehtods}), we can write $T(y)(t)$ by first decomposing $G(y,t,s)$ in a Chebyshev series $\sum_i b_i(t)C_i$ for any $t$, and then integrating it through the mapping $\frak I$. Since $G$ depends both on $y\in X_N$ and $t$, the coefficients $b_i$ are functions of the coefficients $a_i$ that uniquely determine $y$, and the variable $t$, i.e. they are of the form $b_i(a_0, \ldots, a_N,t)$. We want to show now that they are continuous functions $\mathbb R^{N+1}\times [-1,1] \longrightarrow \mathbb R$. To do so, first recall that the Chebyshev coefficients of $G(y,t,s)$ (for fixed $t$) are found through the formula
		$$
		b_k = \frac{2}{\pi}\int_{-1}^{1} \frac{G(y,t,s)C_k(s)}{\sqrt{1-s^2}} ds.
		$$
		As seen above, for all $s\in [-1,1]$ we have $|y_1(s) - y_2(s)| \leq \sum_i |a_{i,1} - a_{i,2}|$ for functions in $X_N$, from which we have that using the continuity of $G$, for any choice of $\epsilon >0$, we can find $\delta >0$ small enough such that $|G(y_1(s),t_1,s) - G(y_2(s),t_2)| < \epsilon /2$ whenever $(a_{0,2}, \ldots , a_{N,2},t_2) \in \mathbb B_\delta (a_{0,1}, \ldots , a_{N,1},t_1)$. Now, we have
		\begin{eqnarray*}
			\lefteqn{|b_k(a_{0,1}, \ldots , a_{N,1},t_1)-b_k(a_{0,2}, \ldots , a_{N,2},t_2)|}\\
			&=& \frac{2}{\pi} |\int_{-1}^1\{G(y_1(s),t_1,s) - G(y_2(s),t_2,s)\}\frac{C_k(s)}{\sqrt{1-s^2}}ds| \\
			&\leq& \frac{2}{\pi}\int_{-1}^1 |G(y_1(s),t_1,s) - G(y_2(s),t_2,s)| \frac{|C_k(s)|}{\sqrt{1-s^2}}ds\\
			&<& \frac{2}{\pi} \frac{\epsilon}{2} \int_{-1}^1\frac{1}{\sqrt{1-s^2}}ds = \epsilon,
		\end{eqnarray*}
		whenever $(a_{0,2}, \ldots , a_{N,2},t_2) \in \mathbb B_\delta (a_{0,1}, \ldots , a_{N,1},t_1)$, showing the continuity of $b_k: \mathbb R^{N+1}\times [-1,1] \longrightarrow \mathbb R$ for any arbitrary $k$. 
		
		Now, from the continuity of the coefficients $b_k$, it follows that we can find neural networks $f_k$ of some depth and width (\cite{Lu}) such that $|b_k - f_k| < \frac{\epsilon}{2(N+1)}$ on $\mathbb R^{N+1} \times [-1,1]$, for $k = 0, \ldots, N+1$. Since the spectral integration creates a linear combination of the coefficients $b_k$, we have that setting $F_N = \sum_i f_i$ the inequality $|\frak I P_N G - \frak I F_N| < \epsilon/2$ holds on all $\mathbb R^{N+1} \times [-1,1]$, i.e. for all $y\in X_N$. Since the $b_k$ are continuous and $T(y)(t)$ is continuous as well (with respect to $a_i$ and $t$), we can find a neural network $g_N: \mathbb R^{N+1} \times [-1,1] \longrightarrow \mathbb R$ which satisfies $|T(y)(t) - \frak I P_N G - g_N|<\epsilon/2$ on $\mathbb R^{N+1} \times [-1,1]$ (again \cite{Lu}). This means that on $\mathbb R^{N+1} \times [-1,1]$ we have $\|P_NT(y) - P_N \frak I F_N(y) - g_N(y)\|_\infty < \epsilon$, and moreover $F_N$ approximates the projected integral operator over $X_N$ with arbitrary precision. To complete, now we apply the projector $P_N$ and since this is continuous, the approximation properties found above are preserved (possibly upon choosing an $\epsilon'$ appropriately in the previous discussion).
	\end{proof}
	
	\begin{remark}
		{\rm 
			We can use other universal approximation results such as \cite{Horn}, instead of \cite{Lu}, to derive the approximation capabilities of spectral integral neural networks. Depending on the results used, some extra assumptions need to be imposed on $F_N$ and $g_N$, but the gain is that simpler architectures might be used. For instance, in \cite{Horn}, one finds that a single layer neural network can locally (i.e. in some closed ball) approximate integral operators. 
		}
	\end{remark}
	
	In the following we consider an operator $T$ which satisfies both the of Theorem~\ref{thm:regular_T}, and of Lemma~\ref{lem:spec_approximation}. 
	\begin{theorem}
		Let $T(y) + f = y$ be an integral equation that admits a unique solution $u\in C^{k,\beta}$. Then, for any $\epsilon>0$ we can find $N>0$, and neural networks $G$, $g$ and $\tilde f$ such that if $\tilde u$ is the solution of $P_N\frak I G(y) + g(y) + \tilde f = y$ which we assume to have solutions for $N$ large enough for any $\tilde f$, we have $\|u-\tilde u\|_\infty < \epsilon$.
	\end{theorem}
	\begin{proof}
		For $N$ large enough, $P_N\frak I G(y) + g(y) + \tilde f = y$ has a solution $\tilde u_N$ and we can write 
		$$
		\|u - \tilde u_N\| \leq \|u - u_N\| + \|u_N - \tilde u_N\|,
		$$
		where $u_N$ is the solution to the projected equation $P_NT(y) + P_Nf = y$ as in the case of Theorem~\ref{thm:regular_T}. Then, we can choose $N$ large enough such that $\|u - u_N\|_\infty \leq  \|u - u_N\| < \epsilon/2$, since $\|u-u_N\| \longrightarrow 0$. From Lemma~\ref{lem:spec_approximation} we can find neural networks $G$ and $g$ such that 
		$$
		\|P_NT(y) - P_N \frak I G(y) - g(y)\|_\infty < \epsilon/4,
		$$
		for all $y$ in $X_N$, the projection space of continuous functions (in particular the projection of the H\"{o}lder space). 	Since $f$ is continuous, we can also choose (via the Universal Approximation Theorem) a neural network $\tilde f$ such that
		$$
		\|P_Nf - \tilde f\|_\infty < \epsilon/4. 
		$$
		Then, from the definition of $u_N$ and $\tilde u_N$ we have 
		$$
		\|u_N - \tilde u_N\|_\infty = \|P_NT(u_N) + P_Nf - P_N \frak I G(y) - g(y) - \tilde f\|_\infty \leq \epsilon/4 + \epsilon/4.
		$$
		Therefore, we find that for $N$ large enough, $\|u - \tilde u_N\|_\infty < \epsilon$, which completes the proof. 
	\end{proof}

	We can also derive the following result, which shows that we can at least locally approximate continuous bounded operators in H\"{o}lder spaces through the methods considered in this article, where $T$ is as above. 
	
	\begin{theorem}
		Let $T$ be an integeral operator (Fredholm or Volterra) defined on the H\"{o}lder space $C^{k,\beta}$. Then, for any choice of $\epsilon >0$, we can find $N>0$ and neural networks $G$ and $g$ for which
		\begin{eqnarray*}
			\|T(y) - P_N\frak I G(y) - g(y)\|_\infty < \epsilon,
		\end{eqnarray*}
		for all $y\in \mathbb B(0,r)$, where $\mathbb B(0,r)$ is an arbitrary radius-$r$ ball centered at zero in $C^{k,\beta}$. 
	\end{theorem}
	\begin{proof}
		For any $y$ and any choice of neural networks $G$ and $g$ we use the projection $P_N$ to obtain
		\begin{eqnarray*}
			\|T(y) - P_N\frak I G(y) - g(y)\|_\infty &=&	\|T(y) - P_NT(y) + P_NT(y) - P_N\frak I G(y) - g(y)\|_\infty\\
			&\leq& \|T(y) - P_NT(y)\|_\infty + \|P_NT(y) - P_N\frak I G(y) - g(y)\|_\infty\\
			&\leq& \frac{M_k(y)}{N^{k+\beta}} + \|P_NT(y) - P_N\frak I G(y) - g(y)\|_\infty.
		\end{eqnarray*}
		Let us now choose an arbitrary $r$ and select $N$ such that $\frac{r}{N^{k+\beta}} < \epsilon/2$.  Now, from Lemma~\ref{lem:spec_approximation} we can find $G$ and $g$ such that $\|P_NT(y) - P_N\frak I G(y) - g(y)\|_\infty< \epsilon/2$, which completes the proof.
	\end{proof}
	
	\begin{remark}
		{\rm 
			A generalization of the previous result to Sobolev spaces, using the results of \cite{Can_Quart}, would be of theoretical interest.
		}
	\end{remark}
	
	\begin{remark}
		{\rm 
			Lastly, we observe that in several cases in practice, one is interested in studying the equation 
			$$
			\lambda T(y) + f = y,
			$$
			where $\lambda$ is a nonzero parameter which is useful to determine When $\lambda$ is small enough, the existence and uniqueness of the solutions of the integral equation, e.g. through fixed point iterations, is guaranteed. See \cite{Topological}, Chapter 3, for examples. This is substantially equivalent to an eigenvalue problem for nonlinear operators (for nonzero eigenvalues), where the definition of eigenvalue of $T$ is defined in the ``naive'' way as in \cite{Topological}. See \cite{CaFuVi} for an overview of spectral theory for nonlinear operators.  
		}
	\end{remark}
	
	Regarding Question~\ref{question3}, an approach similar to the one pursued in \cite{NIDE} can be applied here to show that the gradient descent method can be applied to the architecture proposed in this article to obtain the neural networks $G$ and $g$. However, we do not adapt those methods to our case, but rather content ourselves with showing that gradient descend is possible by means of experimentation. 
	
	\section{Algorithm}\label{sec:algo}
	
	The considerations up to now do not take into account the fact that the approximated (projected) equation needs to be solved using some numerical scheme. Therefore, the theoretical bounds obtained in Section~\ref{sec:Theoretical} assume an exact solution of the projected equation. This is obviously not the case in practice, and we should introduce some numerical solver scheme to obtain the solutions $u_N$. 
	
	Our algorithm, therefore, is based on projecting the equation as described in Section~\ref{sec:Theoretical}, and then solving the corresponding equation. Gradient descent is performed on the error between the solution obtained, and the target function. The function $f$ is used for inizialization of the solver, and it is obtained from the available data. For instance, if the problem is to predict a dynamics from the initial $5$ time points, these $5$ points are used to obtain the function $f$, and solving the projected integral equation gives the predicted (approximated) dynamics. 
	
	We use the Chebyshev coefficients as inputs of our neural networks, and peform integration on the output function (identified with its coefficients) through the spectral integration $\frak I$ described in Section~\ref{sec:Mehtods}. To solve the projected equation, we use a fixed point iteration in the projected coefficient space. While the fixed point iteration is guaranteed to converge under the more restrictive assumptions of contractivity of the operator, we see that in practice such convergence is achieved in our experiments. In more delicate cases, one might introduce extra constraints to guarantee the convergence of the iterations. The method is schematically described in Algorithm~\ref{algo:SNIE}.
	
	\begin{algorithm}[h!]
		\caption{Algorithm for the spectral neural integral equation model.}
		\label{algo:SNIE}
		\begin{algorithmic}[1]
			\Require{Integrand neural network $G_\theta$ and $f$}
			\Comment{Initial model and function $f$ obtained from the available data}
			\Ensure{Trained neural network $G_\theta$} \Comment{Solution of Equation~\eqref{eqn:SNIE1D} projected in $N$-dimensional space}
			\State{Project dataset on $N$-dimensional space $X_N$ using $P_N = \sum_{k=0}^{N-1} \frac{2}{\pi}\int_{-1}^1 \frac{F(s)C_k(s)}{\sqrt{1-s^2}}ds \cdot C_k$}
			\State{Initialize iterations with function $u_0 := f$}
			\While{While $\|u_k - u_{k+1}\| > \tau$} \Comment{Solve Equation~\eqref{eqn:SNIE1D} iteratively in $X_N$}
			\State{$u_{k+1} = \frak IG_\theta( u_k) + f$ }
			\EndWhile
			\State{Solution $u_N$} \Comment{Output of the iterative procedure}
			\State{Convert $u_N$ to an approximate solution in original space, $y_N$} \Comment{Use Chebyshev interpolation}
			\State{Compute $\|y_N - u\|$} \Comment{Evaluate the error between $y_N$ and data $u$}
			\State{Use gradients computed while solving Equation~\eqref{eqn:SNIE1D} to perform gradient descent to optimize $G_\theta$}
		\end{algorithmic}
	\end{algorithm}

	It is possible to extend the algorithm to a higher dimensional version by considering a tensor product of Chebyshev grid points. This procedure has been widely employed in practice in engineering, see e.g. \cite{Cheb_tens} for an example. However, the resulting product polynomials do not satisfy the optimality guarantees of the 1D Chebyshev polynomials. Therefore, while the algorithm extends, mutatis mutandis, to a higher dimensional version, the theoretical analysis does not proceed in a straightforward manner. A more formal approach follows the work of Kowalski \cite{Kow1,Kow2} and Xu \cite{Xu1,Xu2}. In fact, they provide a way of constructing families of multivariate orthogonal polynomials that can serve as the basis for projection, playing the same role as the Chebyshev polynomials in this article for an analogous algorithm in higher dimensions. 
		
	\begin{remark}
			{\rm 
					It is possible to choose the tolerance parameter $\tau$ in Algorithm~\ref{algo:SNIE} based on the Chebyshev interpolation. The parameter $\tau$ controls the error in the coefficients of the Chebyshev basis to represent a target function. Therefore, relating the error on the coefficients to the interpolation error allows to meaningfully fix $\tau$ theoretically. However, we point out that it is usually enough to set $\tau$ based on empirical observation, and fix it as a hyperparameter. 
			}
	\end{remark}

	\section{Experiments}\label{sec:experiments} 
	
	We showcase the algorithm described in Section~\ref{sec:algo} on two classes of experiments. We consider a dataset consisting of solutions of an integral equation solved numerically. The kernel of the integral operator consists of matrices with hyperbolic functions entries in the variables $t$ and $s$, where the latter is the variable of integration.Moreover, in order to test the model's stability with respect to noise, the original integral equation solutions are perturbed by adding gaussian noise.
	The second dataset is a simulated fMRI dataset generated by solving a delayed ODE system that describes the behavior of neurons in the brain. 
	
	We show that the model can predict the dynamics, and show that the memory footprint of the model is very low, considering that it is solving an integral equation for each epoch. Moreover, we perform interpolation experiments to demonstrate that the model is very stable with respect to variations in the temporal domain of train and evaluation. We have chosen datasets with nonlocal behaviors, like integral equations and delay differential equations, to explore the capability of the model to model long range temporal dependencies in the datasets.

	\subsection{Integral Equations Dataset}
	
	In this experiment, the model is initialized using two points for each dynamics. The task is to predict the whole noisy dynamics, which consists of $100$ time points. To test the capability of the model to interpolate, we train our Spectral NIE on a downsampled dataset that contains half of the points of the curves. During evaluation, the model then outputs $100$ points, even though it has been trained on $50$ time points alone. Convergence of the model is considered to be $200$ epochs with no improvement on the validation set. The walltime of the experiment is the time elapsed between the beginning of the training to the convergence, as per the aforementioned criterion. An upper bound of one hour training has been fixed, and models that did not converge within this time are reported as having walltime $> 3600$. 
	
	In Table~\ref{Tab:IE_curve} we have reported two different Spectral NIE models, a large model with more than $100$K paramters, and a small model with only $316$ parameters. For both models, the $MC$ represents the number of points used to perform the Monte Carlo integration used in the interpolation. We see that increasing $MC$ by a $10$-fold factor we obtain a faster convergence time, but higher memory footprint.  
	
	For the other models, we have also tested ANIE (which is an implementation of the original NIE of \cite{ANIE,ANIE_NAT}) with two models, a large one and a small one, with number of parameters comparable to the corresponding Spectral NIE models. We find that for small models, the Spectral NIE model has a significantly better performance than ANIE, with smaller walltime and one order of magnitude better interpolation. For the larger models, we observe that Spectral NIE has better memory footprint and better walltime, but ANIE has a slighlty better overall performance in terms of error and interpolation error. Moreover, we see that the memory footprint of the Spectral NIE grows slowly with respect to the size of the model, since most of the computational expense is due to the Monte Carlo integration used for interpolation, and therefore independent on the model itself. ANIE's complexity, however, grows faster with respect to the size of the model. 
	
	Among the other models that we have tested for comparison, we also have ResNet, LSTM, Nerual Ordinary Differential Equation (its newer Latent ODE variant), and Fourier Neural Operator (FNO). No surprisingly, both the integral equation based models perform significantly better than these models, since they either do not implement nonlocal dynamics or, as in the case of LSTM, their nonlocality properties are limited. 
	
	In Table~\ref{Tab:IE_curve}, Error refers to the mean squared error of the predicted dynamics with respect to the real dynamics when trained over the full time interval (no downsampling). Interpolation error refers to the setup  where the dataset is downsampled for training, and the prediction of the model is performed over the full time interval. The experiment shows that the model is quite stable with respect to changes of time interval between train and evaluation. We think that such stability is due to the use of spectral methods, which automatically give a set of polynomials that can be evaluated on arbitrary points, resulting in good interpolation results. All experiments reported in Table~\ref{Tab:IE_curve} were performed on an Nvidia RTX 3090 GPU. 
	
	\begin{table}[h!]
		\centering
		\resizebox{\textwidth}{!}{\begin{tabular}{|c | c | c | c| c| c|}
				\hline
				Models & Memory (MiB) & Walltime (sec) & Error (MSE) & Interp Error (MSE) & Parameters\\
				\hline
				Spectral NIE (MC = $1K$) & $74$ & $165.21$ & $0.0022\pm 0.0016$ & $0.0028\pm 0.0032$ & $126498$ \\
				\hline 
				Spectral NIE (MC = $10K$)& $324$  & $408.04$  & $0.0032 \pm 0.0022$ & $0.0035 \pm 0.0024$ & $316$ \\
				\hline
				Spectral NIE (MC = $1K$)  & $68$ & $546.50$ & $0.0025\pm 0.0013$ & $0.0026\pm 0.0014$ & $316$\\
				\hline
				ANIE (Large) & $122$ & $564.94$ & $0.0014\pm 0.0002$ & $0.0015\pm 0.0003$  & $135363$\\
				\hline
				ANIE (Small) & $44$ & $1111.74$ & $0.0062\pm 0.0009$ & $0.01175\pm 0.0017$ & $531$\\
				\hline
				ResNet & $42$& $305.47$ & $0.4269\pm 0.0588$ & $0.2619\pm 0.0343$ & $30802$\\
				\hline
				LatentODE & $40$ & $>3600$ & $0.6342\pm 0.2838$ & $0.6290\pm 0.2813$ & $1428$\\
				\hline
				LSTM (init 10) & $102$ & $262.90$ & $0.1785\pm 0.0942$ & NA & $41802$\\
				\hline
				LSTM (init 20) & $122$ & $864.66$ & $0.0492\pm 0.0110$ & $0.1707 \pm 0.1285$ & $41802$\\
				\hline
				FNO1D (init 5) & $258$ & $>3600$ & $0.0292\pm 0.0285$ & $0.0994\pm 0.1207$ & $287843$\\   %time: $3734.00$
				\hline 
				FNO1D (init 10) & $250$ & $3488.40$ & $0.0286\pm 0.0268$ & $0.1541\pm  0.2724$ & $288163$\\
				\hline
		\end{tabular}}
		\caption{Experimental results for the Integral Equations Dataset. Both errors are reported as mean squared errors.}
		\label{Tab:IE_curve}
	\end{table}\

		We report results from a more extensive study of the computational trade-offs between Monte Carlo sampling, convergence walltime, and model accuracy. The experiments were performed on a CPU using a medium-sized model with $67,934$ parameters, and the results are presented in Table~\ref{Tab:tradeoff}. We observe that although larger Monte Carlo sampling initially improves accuracy, the gains quickly saturate as sampling increases, while memory usage continues to grow. Moreover, runtime convergence also increases with larger sampling, since each epoch requires more computation time.
		
			\begin{table}[h!]
			\centering
			\begin{tabular}{|c | c | c |c|}
					\hline
					MC& Memory(MiB) & Walltime (sec) & Error (MSE) \\
					\hline
					$500$ & $86$ & $665$ & $0.0020 \pm 0.0012$ \\
					\hline
					 $1000$ & $133$ & $744$ & $0.0021 \pm 0.0012$\\
					 \hline
					 $2000$ & $235$ & $601$ & $0.0018 \pm 0.0010$\\
					 \hline
					 $4000$ & $429$ & $849$ & $0.0018 \pm 0.0009$\\
					 \hline
					 $8000$ & $815$ & $739$ & $0.0019 \pm 0.0010$\\
					 \hline
			\end{tabular}
			\caption{Experimental results for the Integral Equations Dataset on computational trade-off.}
			\label{Tab:tradeoff}
		\end{table}\
		
		To further determine the stability of the model with respect to irregularly sampled data, we consider an experimental setup where we irregularly sample the trajectories to 70\%, 60\%, 50\%, 40\%, 30\%, and 20\% of their original size during training, and predict the original curve during evaluation. Therefore, this is another interpolation test where the model accesses irregular time stamps of the data. The full dynamics is never seen during training by the model, as it is trained to fit only the available downsampled data. Effectively, the interpolation component of evaluation is a generalization task. When the percentage of sampled data becomes too low, the error increases significantly, which is an expected behavior considering that the task becomes increasingly harder, and for sufficiently low available data the theoretical optimal interpolation error bounds increase. The model shows good interpolation capabilities with respect to irregular downsampling of the data, until the downsampling grows excessively. The results of the experiment are reported in Table~\ref{Tab:IE_irregular}. 
		
	\begin{table}[h!]
		\centering
		\resizebox{\textwidth}{!}{\begin{tabular}{|c | c | c |c|c|c|}
				\hline
				70\% & 60\% & 50\% & 40\%& 30\%& 20\%\\
				\hline
				$0.0022 \pm 0.0018$ & $0.0022 \pm 0.0018$ & $0.022\pm 0.0018$& $0.0022 \pm 0.0014$& $0.0022 \pm 0.0015$ & $0.0167 \pm 0.0069$\\
				\hline
		\end{tabular}}
	\caption{Experimental results for the Integral Equations Dataset on irregularly sampled interpolation task.}
	\label{Tab:IE_irregular}
	\end{table}\

	\subsection{Simulated fMRI Dataset}
	
	We consider now a simulated fMRI dataset. This is generated using the package neurolib \cite{neurolib} through a system of delay differential equations that simulate the stimulus of brain regions. The dynamics is $80$ dimensional, and each dimension refers to a brain region. Moreover, the system has a high degree of nonlocality.
	
	We consider different types of initializations, where we assume that $3$, $4$ or $7$ time points are available for initialization of the model. Clearly, as the number of available points for initialization increases, the models' predictions become more accurate, since the task is simpler.
	
	\begin{table}[h!]
		\centering
		\resizebox{\textwidth}{!}{\begin{tabular}{|c | c | c| c|c|c|}
				\hline
				N init	& Spectral NIE & ANIE & NODE & LSTM & FNO1D\\
				\hline
				3  & $0.0027 \pm 0.0006$  & $0.0090 \pm 0.0092$  & $0.1189\pm 0.0437$ & $0.0871\pm 0.0146$& $0.0794\pm 0.0150$\\
				\hline
				4  & $0.0025 \pm 0.0007$ & $0.0096\pm 0.0097$  &$0.1102\pm 0.0243$ & $0.0819\pm 0.0145$ & $0.0726 \pm 0.0130$\\
				\hline
				7  & $0.0012 \pm 0.0006$ &  $0.0069\pm 0.0113$ &$0.1138\pm0.0302$ & $0.0810\pm0.0136$ & $0.0576\pm 0.0100$\\
				\hline
		\end{tabular}}
		\caption{Experimental results for fMRI simulated dataset: all reported values refer to mean squared error. The models are initialized with varying amount of time points.}
		\label{Tab:fMRI}
	\end{table}\
	
	Table~\ref{Tab:fMRI} shows the results of the experiment, where we have compared the Spectral NIE to ANIE, NODE, LSTM and FNO1D. In all experiments, the number of parameters of the model are comparable, to ensure a more meaningful comparison. Once again, since the dynamics is highly nonlocal, we expect that integral equation based models perform better. This is indeed the case, as shown in the table. Moreover, we see that Spectral NIE performs significantly better than ANIE in this experiment, and the gain in using our spectral approach is not only in computational cost and convergence time, but also in accuracy.

	Example predictions for the Spectral NIE for $3$, $4$ and $7$ initialization points are given in Figure~\ref{fig:fMRI_7init}, Figure~\ref{fig:fMRI_4init} and Figure~\ref{fig:fMRI_3init}, respectively. In all figures, the top dynamics represents the prediction, and the bottom dynamics represents the ground truth. The $\vec{x}$-axis represents all the $80$ brain locations, while the $\vec{y}$-axis refers to the $20$ time points per dynamics. 
	
	\begin{figure}[htb]
		\begin{center}
			\includegraphics[width=5in]{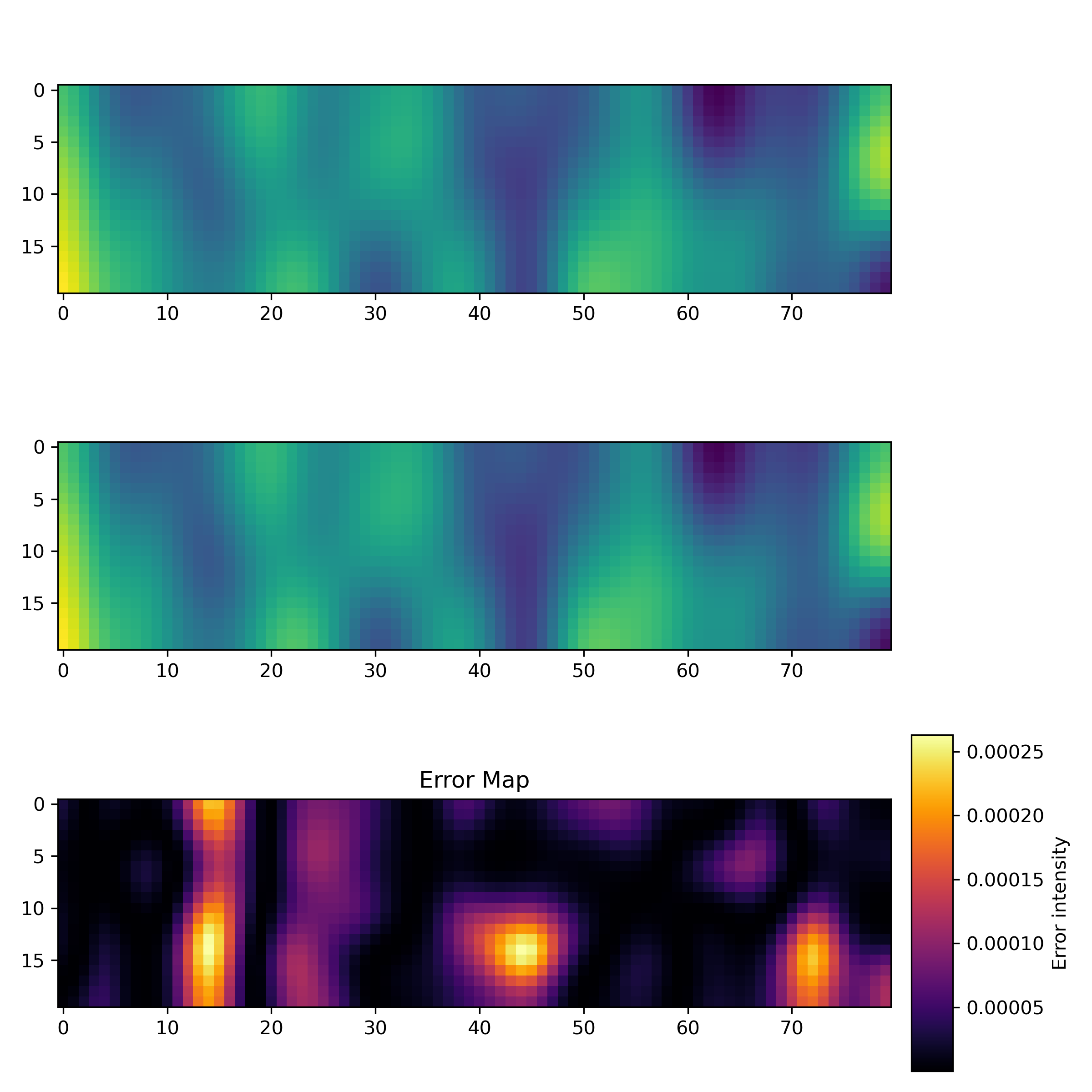}
		\end{center}
		\caption{Example prediction for Spectral NIE on fMRI Dataset with $7$ time points used for initialization. Top: Prediction. Middle: Ground truth. Bottom: Error.}
		\label{fig:fMRI_7init}
	\end{figure}
	
	\begin{figure}[htb]
		\begin{center}
			\includegraphics[width=5in]{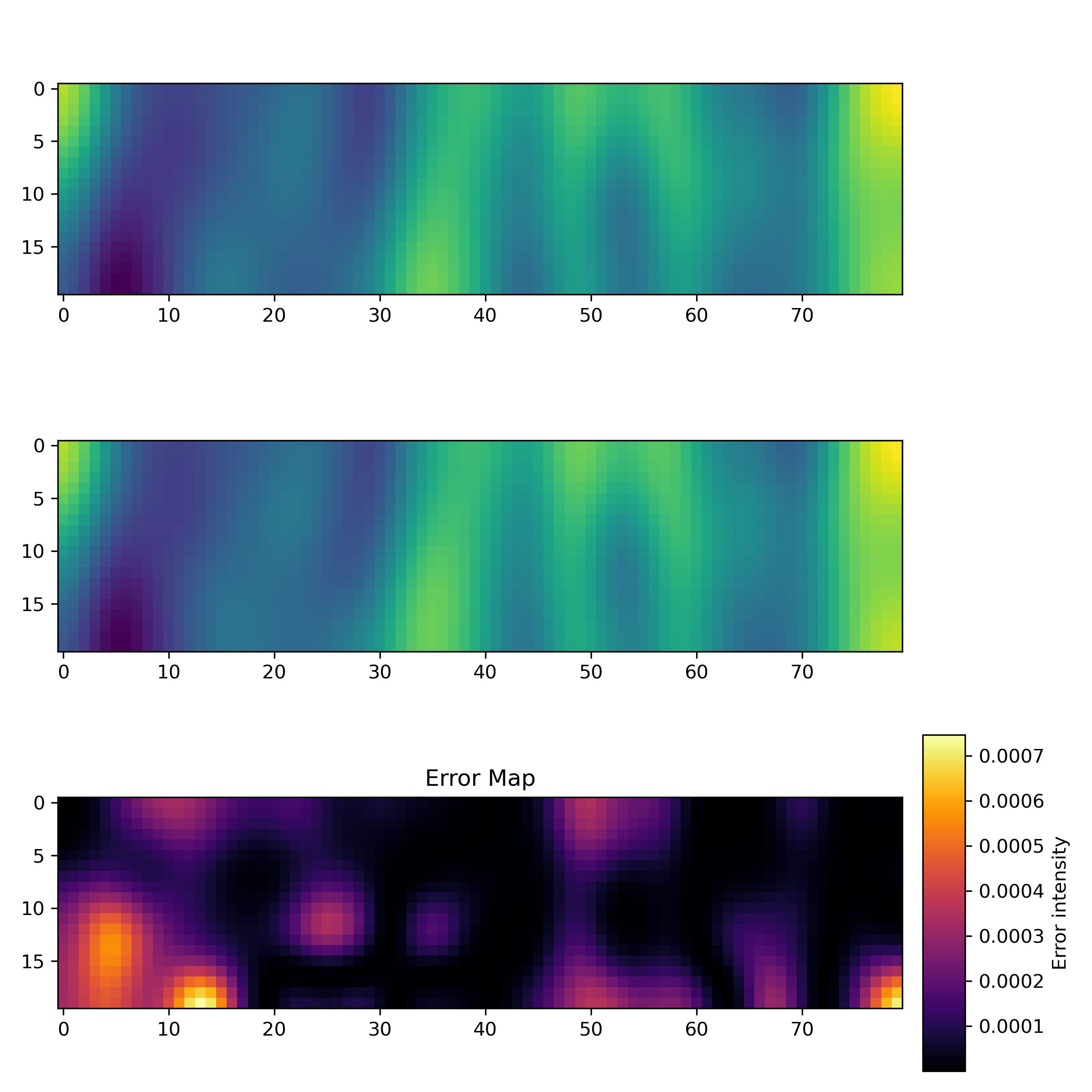}
		\end{center}
		\caption{Example prediction for Spectral NIE on fMRI Dataset with $4$ time points used for initialization. Top: Prediction. Middle: Ground truth. Bottom: Error.}
		\label{fig:fMRI_4init}
	\end{figure}
	
	\begin{figure}[htb]
		\begin{center}
			\includegraphics[width=5in]{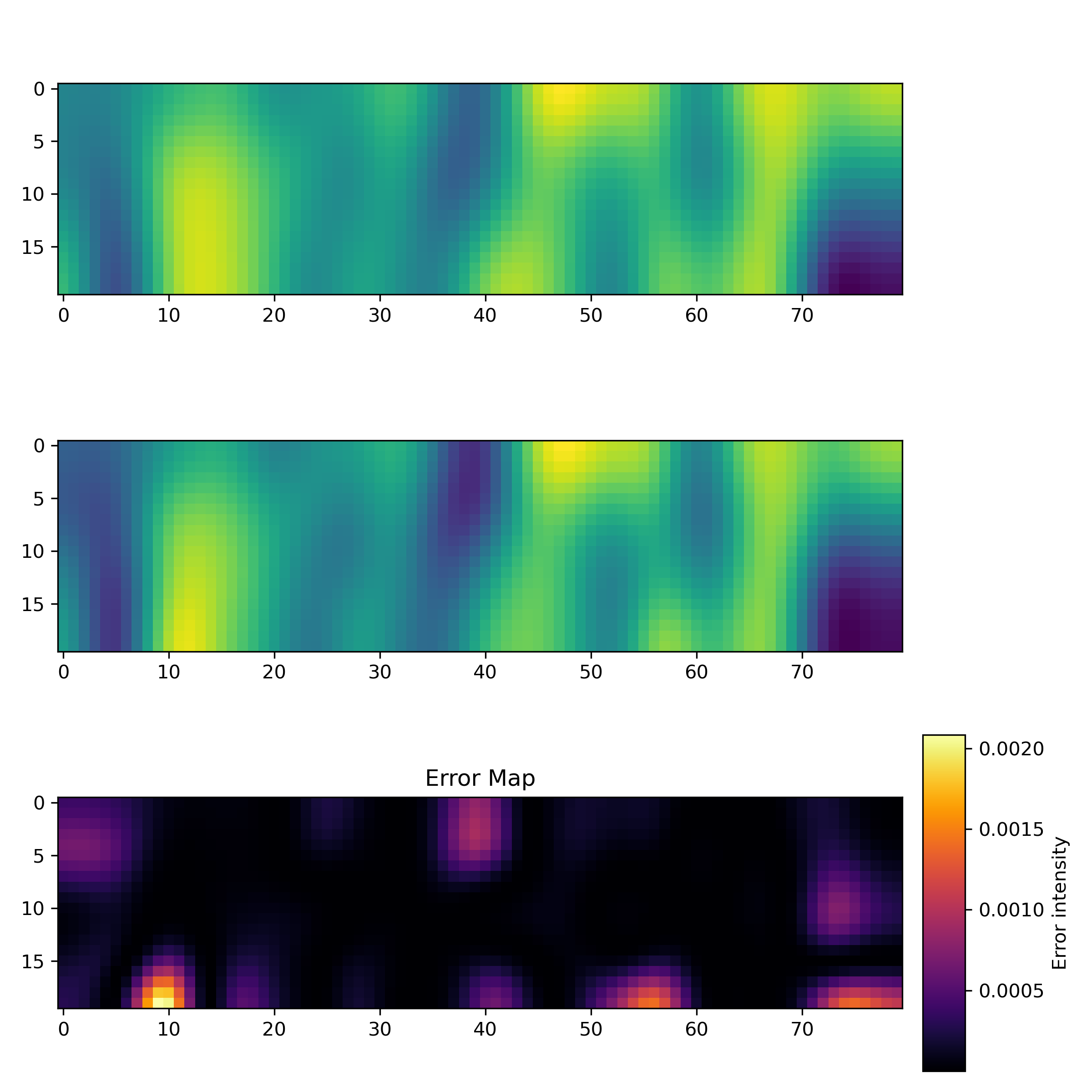}
		\end{center}
		\caption{Example prediction for Spectral NIE on fMRI Dataset with $7$ time points used for initialization. Top: Prediction. Middle: Ground truth. Bottom: Error.}
		\label{fig:fMRI_3init}
	\end{figure}

		We further investigate the runtime-accuracy trade-off on the fMRI dataset, with results summarized in Table~\ref{Tab:tradeoff_fMRI}. All experiments were run on a CPU using a model with $348,648$ parameters. We observe that walltime convergence decreases as the number of Monte Carlo samples increases, before stabilizing. This behavior appears to stem from improved model stability, which allows convergence in fewer training epochs. However, each epoch becomes slower due to the higher computational cost. Overall, accuracy improves markedly with larger Monte Carlo samples.
		
			\begin{table}[h!]
				\centering
				\begin{tabular}{|c | c | c |c|}
					\hline
					MC& Memory(MiB) & Walltime (sec) & Error (MSE) \\
					\hline
					$500$ & $489$ & $431$ & $0.0021 \pm 0.0009$ \\
					\hline
					$1000$ & $503$ & $405$ & $0.0017 \pm 0.0008$\\
					\hline
					$2000$ & $525$ & $369$ & $0.0015 \pm 0.0009$\\
					\hline
					$4000$ & $591$ & $360$ & $0.0013 \pm 0.0006$\\
					\hline
					$8000$ & $602$ & $369$ & $0.0012 \pm 0.0005$\\
					\hline
				\end{tabular}
				\caption{Experimental results for the fMRI Dataset on computational trade-off.}
				\label{Tab:tradeoff_fMRI}
			\end{table}\

			\section{Conclusions}
			
			We conclude by providing some potential applications of the framework introduced in this article. A natural class of problems that this work can be applied to, regards neuroscience. EEG and fMRI have signals that are highly nonlocal, and therefore modeling the corresponding operators in a nonlocal manner, as with integral equations, is meaningful. While fMRI is a higher dimensional method, with 2 or 3 spatial dimensions and a temporal one, it is often the case that only the dynamics of certain brain areas are of interests, leading to a temporal dynamics with one channel per brain area, as in the fMRI dataset considered in this article. 
			Similarly, problems in population dynamics where delay and long-range terms are not negligible lead to nonlocal equations that can be modeled via integral equations. Virus spread modling  represents an example of such situation.  Lastly, 1D Boltzmann equations, or nonlocal diffusion problems where one dimension dominates the equation term are a class of examples of problems in physics that can be treated using the framework in this article.

	\section*{Acknowledgement}
	The author acknowledges support from the NIH under the grant R16GM154734.

\end{document}